\documentclass{amsart}

\usepackage[all]{xypic}
\usepackage{epsf}
\usepackage{verbatim} 
\usepackage{amsmath}
\usepackage{amsfonts}
\usepackage{amssymb}
\usepackage{mathrsfs}
\usepackage{amsthm}
\usepackage{newlfont}

\swapnumbers
 \newtheorem{thm}{Theorem}
 \newtheorem{cor}[thm]{Corollary}
 
 \newtheorem{lem}[thm]{Lemma}
 \newtheorem{prop}[thm]{Proposition}
 
 \theoremstyle{definition}

 \theoremstyle{remark}
 \newtheorem{rem}[thm]{Remark}

 \theoremstyle{remark}
 \newtheorem{prg}[thm]{}

 \theoremstyle{remark}
 \newtheorem{example}[thm]{Example}

 \theoremstyle{plain}
 \newtheorem{question}[thm]{Question}

\numberwithin{equation}{thm}
\numberwithin{thm}{section}

 \newcommand{\trop}{\mathrm{trop}}
 \newcommand{\an}{\mathrm{an}}

 \newcommand{\Hom}{\mathrm{Hom}}
 
 \newcommand{\Spec}{\mathrm{Spec}}

 \newcommand{\End}{\mathrm{End}}
 \newcommand{\Pic}{\mathrm{Pic}}
 
 \newcommand{\ord}{\mathrm{ord}}

 \newcommand{\Gal}{\mathrm{Gal}}
 \newcommand{\GL}{\mathrm{GL}}
 
 \newcommand{\PGL}{\mathrm{PGL}}
 \newcommand{\SL}{\mathrm{SL}}

 \newcommand{\coker}{\mathrm{coker}}

 \newcommand{\ev}{\mathrm{ev}}

 \newcommand{\opp}{\mathrm{opp}}

 \newcommand{\im}{\mathrm{Im}}

 \newcommand{\Tr}{\mathrm{Tr}}

 \newcommand{\fn}{\mathfrak n}
 \newcommand{\fm}{\mathfrak m}

 \newcommand{\fT}{\mathfrak T}

 \renewcommand{\cH}{\mathcal{H}}

 \newcommand{\cT}{\mathcal{T}}
 \newcommand{\cX}{\mathcal{X}}

\newcommand{\sX}{\mathscr{X}}
\newcommand{\sA}{\mathscr{A}}
\newcommand{\sE}{\mathscr{E}}

\newcommand{\sB}{\mathscr{B}}
 \newcommand{\gm}{\mathbb{G}}
 \newcommand{\gmk}{\mathbb{G}_{m, K}^\an}
 \newcommand{\R}{\mathbb{R}}
 \newcommand{\C}{\mathbb{C}}
 \newcommand{\F}{\mathbb{F}}
 \newcommand{\Q}{\mathbb{Q}}
 \newcommand{\Z}{\mathbb{Z}}

 \newcommand{\T}{\mathbb{T}}
 \newcommand{\M}{\mathbb{M}}

 \newcommand{\G}{\Gamma}
 \newcommand{\To}{\longrightarrow}
 \newcommand{\bs}{\setminus}

 \newcommand{\La}{\Lambda}
 \newcommand{\la}{\lambda}

\newcommand\isom{\ensuremath{\overset{\sim}{\longrightarrow}}}
\newcommand\angles[1]{\ensuremath{\langle#1\rangle}}
\newenvironment{sm}{\left[\begin{smallmatrix}}
                   {\end{smallmatrix}\right]}
\newcommand\smallmat[4]%
  {\ensuremath{\begin{sm}#1&#2\\#3&#4\end{sm}}}
\newcommand\inject{\ensuremath{\hookrightarrow}}
\newcommand\surject{\ensuremath{\twoheadrightarrow}}
\newcommand\tensor\otimes

\renewcommand\bar\overline

\begin{document}

\title[Optimal quotients and component groups]{Optimal quotients of Jacobians with toric reduction and component groups}

\author{Mihran Papikian and Joseph Rabinoff}

\address{Department of Mathematics, Pennsylvania State University, University Park, PA 16802}
\email{papikian@math.psu.edu}
\thanks{The first author was supported in part by Simons Foundation.} 

\address{School of Mathematics, Georgia Institute of Technology, 686 Cherry Street, Atlanta, GA 30332}
\email{rabinoff@math.gatech.edu}




\begin{abstract} Let $J$ be a Jacobian variety with toric reduction over a local field $K$.  
Let $J \to E$ be an optimal quotient defined over $K$, where $E$ is an elliptic curve. 
We give examples in which the functorially induced map $\Phi_J \to \Phi_E$ 
on component groups of the N\'eron models is not surjective. 
This answers a question of Ribet and Takahashi. 
We also give various criteria under which $\Phi_J \to \Phi_E$  is surjective, and discuss 
when these criteria hold for the Jacobians of modular curves.  
\end{abstract}

\maketitle

\section{Introduction}

Let $J$ be the Jacobian variety of a smooth, projective, geometrically irreducible curve defined over a field $K$. 
An \textit{optimal quotient} of $J$ is an abelian variety $E$ over $K$ and a smooth surjective 
morphism $\pi: J\to E$ whose kernel is connected, i.e., an abelian variety; cf. \cite[Def. 3.1]{CS}. 
From now on we assume that $E$ is an elliptic curve and $K$ is a local field. 
The following question, originally posed by Ribet and Takahashi, appears in \cite{BakerLetter}: 

\begin{question}\label{qR} Assume $J$ has (purely) toric reduction; see $($\ref{TorRed}$)$ for the definition. 
Is the functorially induced map 
$\pi_\ast: \Phi_J(\bar{k})\to \Phi_E(\bar{k})$ on component groups of the N\'eron models of 
$J$ and $E$ necessarily surjective, where $\bar{k}$ is the algebraic closure of the residue field of $K$? 
\end{question} 

In Section \ref{sec2}, we will construct examples which show that the answer is No, contrary to
the expectation expressed in \cite{BakerLetter}.  
The interest in Question \ref{qR} comes from arithmetic geometry, where for certain modular Jacobians, such as $J_0(p)$ 
over $\Q_p$, the answer was known to be positive; see Section \ref{sMJ}. It is natural then to ask whether the surjectivity of 
the map on component groups is a general geometric property of Jacobians with toric reduction, 
or is a special arithmetic property of modular Jacobians with toric reduction. Our examples 
indicate that the latter is the case. Of course, Question \ref{qR} makes perfect sense 
without assuming that $J$ has toric reduction, but the answer to 
that more general question was known to be negative even for the  
modular Jacobians $J_0(N)$ of small level. The following example is due to William Stein: 

\begin{example}
There is a unique weight-$2$ newform of level $33$ with integer Fourier coefficients, and the 
corresponding optimal quotient of $J_0(33)$ is the elliptic curve $E:y^2+xy=x^3+x^2-11x$. 
Consider the optimal quotient $\pi: J_0(33)\to E$ over $\Q_3$. The reduction of $J_0(33)$ 
over $\Q_3$ is semi-stable but not toric. By \cite[p. 174]{MazurEisen}, 
$\Phi_{J_0(33)}(\overline{\F}_3)\cong \Z/2\Z$. 
On the other hand, $\Phi_E(\overline{\F}_3)\cong \Z/6\Z$, so $\pi_\ast$ is not surjective.  
\end{example}

The idea of our construction giving a negative answer to Question \ref{qR} 
is to take two elliptic curves $E_1$ and $E_2$ over $K$ with multiplicative reduction and non-trivial component groups. 
We show that one can choose a finite subgroup-scheme $G$ of the abelian surface $E_1\times E_2$ such 
that the quotient $J=(E_1\times E_2)/G$ is a Jacobian variety and $\Phi_J=1$. 
Moreover, $E_1$ and $E_2$ are optimal quotients of $J$. 
Due to (\ref{IsogTorRed}), $J$ automatically has toric reduction. 
Clearly the corresponding maps on component groups cannot be surjective. 
The study of Jacobians isogenous to a product of two elliptic curves has a long history, dating back to Legendre and Jacobi. 
In more recent times such Jacobians have found applications in a variety of arithmetic problems, 
for example, the construction of curves with a maximal number of rational points over 
finite fields \cite{SerreHarvard}, or the construction of Jacobians over $\Q$ with 
large rational torsion subgroups \cite{HLP}. 

From the work of Gerittzen, Mumford and others it is known that abelian varieties 
with toric reduction have rigid-analytic uniformizations. 
(In fact any abelian variety has such a uniformization, but we will only
be concerned with the totally degenerate case.)
In Section \ref{sec3}, we investigate 
the map $\pi_\ast: \Phi_J\to \Phi_E$ using analytic techniques. Some of our arguments here 
are inspired by \cite{Analytic}, \cite{RibetLetter} and \cite{Zagier}. 
We show that the Tate period of 
$E$ can be obtained from $J$ via a natural evaluation map. In this construction,
which is a generalization of the constructions due to Gekeler and Reversat \cite{GR}, 
Bertolini and Darmon \cite{BD}, and Takahashi \cite{TakahashiIJNT},  
the uniformizing lattice of $J$ maps to a subgroup in $K^\times$ isomorphic to $\Z/c\Z\oplus \Z$. 
We show that the cokernel of $\pi_\ast$ is isomorphic to $\Z/c\Z$. 
We also show that $c$ is closely related to the denominator of the idemponent in $\End(J)\otimes \Q$ corresponding 
to $E$. These results are of independent interest, and could be useful in the theory 
of Mumford curves. The main theorem of this section is Theorem \ref{thm:when.surj}, which
gives equivalent conditions for $\pi_\ast$ to be surjective. One of these conditions 
shows that Question \ref{qR} can be interpreted as an analogue for Mumford curves of the problem of 
the equality of the degree of modular parametrization of an elliptic curve over $\Q$ 
and the congruence number of the corresponding newform; see Remark \ref{remModDeg}. 
At the end of Section \ref{sec3}, we give two additional criteria for $\pi_\ast$ being surjective, which are   
based on an assumption that $\End(J)$ contains a subring with certain properties; see 
Lemmas \ref{lemGek} and \ref{lemCII}.   

In Section \ref{sMJ}, we discuss Question \ref{qR} in the context of Jacobians of modular curves. 
We show that this question has positive answer for 
\begin{itemize}
\item $J_0(p)$ considered over $\Q_p$ (see Theorem \ref{thmJ0p}), 
\item the Jacobian of Drinfeld modular curve $X_0(\fn)$ of arbitrary level $\fn\in \F_q[T]$ considered over $\F_q(\!(1/T)\!)$ 
(see Theorem \ref{thmJDM}).
\end{itemize}
(Theorem  \ref{thmJ0p} was known, but we give a different proof which relies on Lemma \ref{lemCII}.)
In this section we also point out a mistake in the published literature. 
Let $J^D_0(M)$ be the Jacobian of the Shimura curve over $\Q$ associated with an Eichler order of level $M$ 
in an indefinite quaternion algebra over $\Q$ of discriminant $D$. 
For any prime $p$ dividing $D$ the Jacobian $J^D_0(M)$ has toric reduction over $\Q_p$. 
Theorem 2.4 in \cite{Takahashi} claims that $\pi_\ast$ is surjective in this case. 
The proof of this theorem crucially relies on a result of Bertolini and Darmon \cite[Prop. 4.4]{BD}. Unfortunately, 
the proof of this latter proposition has a gap, cf. (\ref{prgBD}), 
so Question \ref{qR} in this case remains a very interesting open problem. Theorem \ref{thm:when.surj} 
could be useful for a computational investigation of this problem; see Remark \ref{rem:computations}. 


\section{N\'eron models} 
For the convenience of the reader and future reference we collect in this section some 
terminology and facts about abelian varieties and their N\'eron models. 
The standard reference for the theory of N\'eron models is \cite{NM}. 

\begin{prg}
From now on, $K$ will be a field equipped with a nontrivial discrete valuation 
$$
\ord_K: K\to \Z\cup \{+\infty\}.
$$ 
Let $R=\{z\in K\ |\ \ord_K(z)\geq 0\}$ 
be its ring of integers. Let $\fm=\{z\in K\ |\ \ord_K(z)>0\}$ be the 
maximal ideal of $R$, and $k=R/\fm$ be the residue field. We fix a 
uniformizer $\varpi$ of $R$, and assume that the valuation is normalized by $\ord_K(\varpi)=1$. 
Assume further that $k$ is a finite field of characteristic $p$, 
and define the non-archimedean absolute value on $K$ by $|x|=(\# k)^{-\ord_K(x)}$. 
Finally, assume $K$ is complete for the topology defined by this absolute value. Overall,  
our assumptions mean that $K$ is a \textit{local field}. 
It is known that every local field is isomorphic either to a 
finite extension of $\Q_p$, or to the 
field of formal Laurent series $k(\!(x)\!)$.  
We denote by $\C_K$ the completion of an algebraic closure $\bar{K}$ of $K$ with respect to the extension of the absolute value 
(which is itself algebraically closed). 
\end{prg}

\begin{prg}
If $X$ is a scheme over the base $S$ and $T\to S$ is any base change, $X_T$ will denote 
the pullback of $X$ to $T$. If $T=\Spec(A)$, we may also denote this scheme by $X_A$. 
By $X(T)$ we mean the $T$-rational points of the $S$-scheme $X$, and again, if $T=\Spec(A)$, 
we may also denote this set by $X(A)$. 
\end{prg}

\begin{prg}\label{NerMod}
Let $X$ be a scheme over $K$. A \textit{model} of $X$ over $R$ is an $R$-scheme $\sX$ 
such that its generic fiber $\sX_K$ is isomorphic to $X$. Let $A$ be an abelian variety over $K$. 
There is a model $\sA$ of $A$ which is smooth, separated, and of finite type 
over $R$, and which satisfies the following universal property:
For each smooth $R$-scheme $\sX$ and each $K$-morphism $\phi_K: \sX_K\to A$ 
there is a unique $R$-morphism $\phi:\sX\to \sA$ extending $\phi_K$.   
The model $\sA$ is called the \textit{N\'eron model} of $A$. 
It is obvious from the universal property that $\sA$ is uniquely determined by $A$, up to 
unique isomorphism.  Moreover, the group scheme structure of $A$ uniquely extends to 
a commutative $R$-group scheme structure on $\sA$, and $A(K)=\sA(R)$. 
\end{prg}

\begin{prg} The closed fibre $\sA_k$ is usually not connected. Let $\sA_k^0$ 
be the connected component of the identity section. There is an exact sequence 
$$
0\to \sA_k^0\to \sA_k\to \Phi_A\to 0,
$$
where $\Phi_A$ is a finite \'etale group scheme over $k$. The group $\Phi_A$ 
is called the \textit{group of connected components} of $A$. 
\end{prg}

\begin{prg}\label{prg:MapCG} Let $f_K: A\to B$ be a morphism of abelian varieties. 
By the N\'eron mapping property, the morphism $f_K$ extends 
to a homomorphism $f: \sA\to \sB$. Restricting to the closed fibres we get a homomorphism 
$f_k: \sA_k\to \sB_k$. This homomorphism maps $\sA_k^0$ into $\sB_k^0$. Hence 
there are induced homomorphisms $f_k^0: \sA_k^0\to \sB_k^0$ and $f_\ast: \Phi_A\to \Phi_B$. 
We say that $f_\ast$ is surjective, 
if the homomorphism of abelian groups $f_\ast: \Phi_A(\bar{k})\to \Phi_B(\bar{k})$ is surjective. 
\end{prg}

\begin{prg}\label{UnRam} Let $K'$ be an unramified extension of $K$. Let $R'$ be the ring of integers of $K'$. 
Let $f_{K'}:A_{K'}\to B_{K'}$ be the base change of $f_K$ 
to $K'$. Then $f\otimes R': \sA\otimes_R R' \to \sB\otimes_R R'$ is the corresponding morphism of the 
N\'eron models; see \cite[Cor. 7.2/2]{NM}. This implies that 
$f_\ast: \Phi_{A}(\bar{k})\to \Phi_B(\bar{k})$ does not change under unramified field 
extensions of $K$.  
\end{prg}

\begin{prg}\label{TorRed} By a theorem of Chevalley, 
$\sA_k^0$ is uniquely an extension of an abelian variety $B$ by a connected affine 
group $T\times U$ over $k$, where $T$ is an algebraic torus and $U$ is a unipotent algebraic group; 
see \cite[$\S$9.2]{NM}. We say that $A$ has 
\begin{enumerate}
\item \textit{good reduction} if $U$ and $T$ are trivial,
\item \textit{semi-stable reduction} if $U$ is trivial,
\item \textit{toric reduction} if $U$ and $B$ are trivial, 
\item \textit{split toric reduction} if $U$ and $B$ are trivial, and $T$ is a split torus over $k$.  
\end{enumerate}
Some authors say that $A$ has \textit{purely} toric reduction over $K$ when $U$ and $B$ are trivial. 
If $A$ is an elliptic curve, then it is more common to say that $A$ has \textit{multiplicative} (resp.\ \textit{split multiplicative}) 
reduction over $K$, instead of toric (resp.\ split toric) reduction. 
\end{prg}

\begin{prg}\label{IsogTorRed}
If $A$ has toric reduction, and $f_K: A\to B$ is an isogeny, then $f_k^0$ is an isogeny; cf. \cite[Cor. 7.3/7]{NM}. 
This implies that $B$ also has toric reduction. 
If $f_K: B\to A$ is a closed immersion of abelian varieties and $A$ has toric reduction, then 
$f_k^0$ is a closed immersion; see the proof of Theorem 8.2 in \cite{CS}. 
This implies that if $A$ has (split) toric reduction, then any abelian subvariety of $A$ 
also has (split) toric reduction. Denote by $A^\vee$ and $B^\vee$ the abelian 
varieties dual to $A$ and $B$, respectively. Then $f_K$ is an optimal quotient if and only if 
the dual morphism $f_K^\vee: B^\vee\to A^\vee$ is a 
closed immersion; cf. \cite[Prop. 3.3]{CS}.  
\end{prg}

\section{Rigid-analytic constructions}\label{sec3}

First, we briefly review some facts from the 
theory of rigid-analytic uniformization of abelian varieties. The  
abelian varieties in this section are assumed to have split toric reduction over $K$. 
Since an abelian variety with toric reduction acquires split toric reduction 
over an unramified extension of $K$, as far as the questions of surjectivity 
of the maps of component groups are concerned, the assumption that the reduction is split 
is not restrictive; cf. (\ref{UnRam}). 

\begin{prg} Let $\fT:=(\gm_{m, K}^g)^\an$ be the rigid-analytification of $$\gm_{m, K}^g=\Spec K[Z_1, Z_1^{-1}, \dots, Z_g,
Z_g^{-1}].$$ A \textit{character} of $\fT$ is a homomorphism 
of rigid-analytic groups $\chi: \fT\to \gmk$. Denote the group of characters of $\fT$ 
by $\cX(\fT)$. It is known that analytic characters are all algebraic: 
$$
\cX(\fT)=\{Z_1^{n_1}\cdots Z_g^{n_g}\ |\ (n_1,\dots, n_g)\in \Z^g\}.
$$
In fact, a stronger statement is true: any holomorphic, nowhere vanishing function on $\fT$ is a 
constant multiple of an algebraic character (see \cite[$\S$6.3]{FvdP}). 

Consider the group homomorphism 
\begin{align*}
\trop: \fT(\C_K) &\to \Hom(\cX(\fT), \R)\approx \R^g\\ 
x &\mapsto (\chi\mapsto -\log|\chi(x)|). 
\end{align*}
A (split) \textit{lattice} $\La$ in $\fT$ is a free rank-$g$ subgroup of $\fT(K)$ 
such that $\trop:\La\to \R^g$ is injective and its image is a lattice in the classical sense. 
Such $\La$ is discrete in $\fT$, i.e., the intersection of $\La$ with any affinoid subset of 
$\fT$ is finite. Hence we can form the quotient $\fT/\La$ 
in the usual way by gluing the $\La$-translates of a small enough affinoid. 
The Riemann form condition in this setting is the following:
\end{prg}
\begin{thm}\label{thmRF}
$\fT/\La$ is isomorphic to the rigid-analytification of an abelian variety over $K$
if and only if there is a homomorphism
$$
H:\La\to \cX(\fT)
$$
such that $H(\la)(\mu)=H(\mu)(\la)$ for all $\la,\mu\in \La$, and the symmetric bilinear
form 
\begin{align*}
\langle \cdot, \cdot \rangle_H: \La\times \La &\to \Z \\
\la, \mu &\mapsto \ord_K H(\la)(\mu)
\end{align*}
is positive definite.
\end{thm}
\begin{proof}
See \cite[Ch. 6]{FvdP} or \cite[$\S$2]{BL}. 
\end{proof}

\begin{prg} Let $A$ be an abelian variety of dimension $g$ defined over $K$. 
We say that $A$ is \textit{uniformizable by a torus} if $A^\an\cong \fT/\La$ 
for some lattice $\La$. 
\end{prg}

\begin{thm}\label{thmUnif}
An abelian variety over $K$ is uniformizable by a torus if and only if it has split toric reduction. 
\end{thm}
\begin{proof}
See \cite[$\S$1]{BL}.  
\end{proof}

\begin{prg}\label{polarization} If 
$A$ has split toric reduction, then $A^\vee$ also has split toric reduction; cf. (\ref{IsogTorRed}). 
Let $\fT/\La$ be the uniformization of $A$. Denote 
$$
\fT^\vee=\Hom(\La, \gmk)\quad \text{and}\quad \La^\vee=\Hom(\fT, \gmk). 
$$
Note that $\La^\vee$ is the group of characters $\cX(\fT)$. We have a natural bilinear pairing 
$\La^\vee\times \fT(K)\to K^\times$ given by evaluation of characters on the points of $\fT$. 
For a fixed $\la'\in \La^\vee$, this pairing induces by restriction a 
homomorphism $\La\to K^\times$, $\la\mapsto \la'(\la)$, 
and hence a $K$-valued point in $\fT^\vee$. If we vary $\la'\in \La^\vee$, we obtain a 
canonical homomorphism $\La^\vee\to \fT^\vee$, which is easy to see is the 
dual of $\La\to \fT$.  Hence $\La^\vee$ is naturally a lattice in $\fT^\vee$, and we 
can form the quotient $\fT^\vee/\La^\vee$ as a proper rigid-analytic group. 
As one might expect, $\fT^\vee/\La^\vee$ is 
canonically isomorphic to $(A^\vee)^\an$; see \cite[Thm. 2.1]{BL}. 
Let $H: \La\to \La^\vee$ be a Riemann form for $A$. Applying $\Hom(\ \cdot\ , \gmk)$ to $H$, we 
get a surjective homomorphism $H_\fT: \fT\to \fT^\vee$. From the definitions it is easy to see that 
the restriction of $H_\fT$ to $\La\subset \fT$ is $H$. Hence we get a homomorphism 
$H_{A^\an}: A^\an\to (A^\vee)^\an$. By GAGA, $H_{A^\an}$ canonically corresponds to a 
homomorphism $H_A: A\to A^\vee$. 
Since $H$ is injective with finite cokernel, $H_A$ is an isogeny. 
In fact, one can show that $H_A$ is a polarization and every polarization arises in this manner; cf. \cite[$\S$2]{BL}. 
\end{prg}

\begin{prg} \label{polarization.symmetric}
  More symmetrically, let $\Lambda$ and $\Lambda^\vee$ be two finitely
  generated free abelian groups of the same rank and let
  $[\cdot,\cdot]:\Lambda\times\Lambda^\vee\to K^\times$ be a bilinear
  pairing such that the pairing
  \[ \angles{\cdot,\cdot} = \ord_K\circ[\cdot,\cdot] ~:~
  \Lambda\times\Lambda^\vee\to\Z \]
  becomes perfect after extending scalars from $\Z$ to $\R$.  
  Let $\fT = \Hom(\Lambda^\vee,\gmk)$ and 
  $\fT^\vee = \Hom(\Lambda,\gmk)$.  Then $[\cdot,\cdot]$ defines injective
  homomorphisms $\Lambda\inject\fT(K)$ and
  $\Lambda^\vee\inject\fT^\vee(K)$, the images of which are lattices.  
  With these notations, a Riemann form is a homomorphism
  $H:\Lambda\to\Lambda^\vee$ such that 
  $[\cdot,\cdot]_H = [\cdot,H(\cdot)]$ is symmetric and 
  $\angles{\cdot,\cdot}_H = \angles{\cdot,H(\cdot)}$ is
  positive-definite.  If such a form exists, then 
  $\fT/\Lambda$ and $\fT^\vee/\Lambda^\vee$ are dual abelian varieties.

\end{prg}

\begin{prg}\label{eqFunct} Let $A_1^\an =\fT_1/\La_1$ and $A_2^\an=\fT_2/\La_2$ be uniformizable abelian varieties.  
Let $\Hom(\fT_1, \La_1; \fT_2, \La_2)$ denote the group of homomorphisms $\varphi: \fT_1\to \fT_2$ 
of analytic tori such that $\varphi(\La_1)\subset \La_2$. By a result of Gerritzen  \cite{Gerritzen1}, the natural map 
\begin{equation*}
\Hom(\fT_1, \La_1; \fT_2, \La_2) \to \Hom(A_1, A_2)
\end{equation*}
is a bijection (see also \cite[$\S$7]{SGA7}).

Following the notations in (\ref{polarization.symmetric}), for $i=1,2$ let
$\Lambda_i^\vee = \cX(\fT_i)$, let $\fT_i^\vee$ be the torus with
character lattice $\Lambda_i$, let
$[\cdot,\cdot]_i:\Lambda_i\times\Lambda_i^\vee\to K^\times$ denote the
pairing induced by the inclusion $\Lambda_i\inject\fT_i(K)$, and let
$\angles{\cdot,\cdot}_i = \ord\circ[\cdot,\cdot]_i$.  Let
$\varphi\in\Hom(\fT_1,\Lambda_1;\fT_2,\Lambda_2)$.  Then
$\varphi$ is determined by the induced homomorphism
$\varphi^\vee: \Lambda_2^\vee\to\Lambda_1^\vee$ of character groups, and
since $\varphi(\Lambda_1)\subset\Lambda_2$, we have
\begin{equation}  \label{eq:homom.uniformized}
  [\varphi(\lambda_1),\, \lambda_2^\vee]_2 = 
  [\lambda_1,\, \varphi^\vee(\lambda_2^\vee)]_1 
\end{equation}
for all $\lambda_1\in\Lambda_1$ and 
$\lambda_2^\vee\in\Lambda_2^\vee$.  We can therefore define 
$\Hom(\fT_1,\Lambda_1;\fT_2,\Lambda_2)$ more symmetrically as the group
of pairs $(\varphi,\varphi^\vee)$ of homomorphisms 
$\varphi:\Lambda_1\to\Lambda_2$ and
$\varphi^\vee:\Lambda_2^\vee\to\Lambda_1^\vee$ 
satisfying~\eqref{eq:homom.uniformized}. 
Since $\angles{\cdot,\cdot}_i$ is nondegenerate for $i=1,2$, it is clear
that $\varphi$ and $\varphi^\vee$ determine each other.  If
$(\varphi,\varphi^\vee)\in\Hom(\fT_1,\Lambda_1;\fT_2,\Lambda_2)$
corresponds to the homomorphism $f:A_1\to A_2$ then 
$(\varphi^\vee,\varphi)\in\Hom(\fT_2^\vee,\Lambda_2^\vee;\fT_1^\vee,\Lambda_1^\vee)$
corresponds to the dual homomorphism $f^\vee:A_2^\vee\to A_1^\vee$.

Now let $H_i:\Lambda_i\isom\Lambda_i^\vee$ be Riemann forms determining
principal polarizations $A_i\isom A_i^\vee$ for $i=1,2$.  Using
$H_i$ to identify $\Lambda_i$ with $\Lambda_i^\vee$, we can describe an
element of $\Hom(\fT_1,\Lambda_1;\fT_2,\Lambda_2)$ as a pair
$(\varphi,\varphi^\vee)$, where $\varphi:\Lambda_1\to\Lambda_2$ and
$\varphi^\vee:\Lambda_2\to\Lambda_1$ are homomorphisms
satisfying
\begin{equation}  \label{eq:homom.principal}
  [\varphi(\lambda_1),\, \lambda_2]_{H_2} = 
  [\lambda_1,\, \varphi^\vee(\lambda_2)]_{H_1} 
\end{equation}
for all $\lambda_1\in\Lambda_1$ and $\lambda_2\in\Lambda_2$.  As above, if
$(\varphi,\varphi^\vee)$ corresponds to the homomorphism $f:A_1\to A_2$ then
$(\varphi^\vee,\varphi)$ corresponds to the dual homomorphism
$f^\vee:A_2\cong A_2^\vee\to A_1^\vee\cong A_1$. 

\end{prg}

\begin{prop}\label{prop3.3} 
Assume $A^\an\cong \fT/\La$ is a principally polarizable abelian variety. 
Fix a principal polarization $H:\La\xrightarrow{\sim}\cX(\fT)$. An endomorphism 
$T\in \End(A)$ induces an endomorphism of $\La$, which we denote by the same letter. 
Let $T^\dag\in \End(A)$ be the image of $T$ under the Rosati involution 
with respect to the principal polarization $H$. Then for any $\la, \mu\in \La$, 
\begin{equation*}
H(T\la)(\mu)=H(\la)(T^\dag \mu). 
\end{equation*}
\end{prop}
\begin{proof} 
Let $\Lambda^\vee = \cX(\fT)$ and let
$[\cdot,\cdot]:\Lambda\times\Lambda^\vee\to K^\times$ be the pairing
induced by the inclusion $\Lambda\inject\fT(K)$, as
in~(\ref{polarization.symmetric}).  By (\ref{eqFunct}), we can describe
$T$ as a pair of endomorphisms
$\varphi,\varphi^\vee: \Lambda\to\Lambda$ satisfying 
\[ H(\varphi(\lambda))(\lambda') 
= [\varphi(\lambda),\,\lambda']_H = [\lambda,\,\varphi^\vee(\lambda')]_H 
= H(\lambda)(\varphi^\vee(\lambda')) \]
for all $\lambda,\lambda'\in\Lambda$.  The endomorphism
$T^\dag$ then corresponds to the pair $(\varphi^\vee,\varphi)$.  Under
these identifications, the endomorphism of $\Lambda$ induced by $T$
(resp.\ $T^\dag$) is exactly $\varphi$ (resp.\ $\varphi^\vee$).
\end{proof}

\begin{prg}
Let $J:=\Pic^0_{X/K}$ be the Jacobian variety of a smooth, projective, geometrically irreducible  
curve $X$ over $K$. Assume 
$J$ has split toric reduction; this is equivalent to $X$ being a Mumford curve. 
Let $H$ be the canonical principal polarization on $J$.  The uniformization of 
$J$ is given by 
\begin{equation*}
0\to \La\xrightarrow{H}\Hom(\La, \C_K^\times)\to J(\C_K)\to 0. 
\end{equation*}

Let $E$ be an elliptic curve which is an optimal quotient $\pi:J\to E$. 
Using the canonical principal polarizations on $E$ and $J$, 
we can consider $E$ as an abelian subvariety of $J$ via the dual morphism $\pi^\vee: E\hookrightarrow J$; cf. (\ref{IsogTorRed}). 
Sometimes to emphasize that we consider $E$ as the image of $\pi$ (resp.\ the domain of $\pi^\vee$) we 
will write $E_\ast$ (resp.\ $E^\ast$). 
\end{prg}

To simplify the notation, we will 
denote the pairing $\langle\cdot, \cdot \rangle_H$ of Theorem \ref{thmRF} for the canonical 
principal polarization on $J$ by $\langle\cdot, \cdot \rangle$.  Likewise
we denote the pairing $[\cdot,\cdot]_H:\Lambda\times\Lambda\to K^\times$
of~(\ref{polarization.symmetric}) by $[\cdot,\cdot]$.

\begin{prg}Since $E$ is a subvariety of $J$, it has split toric reduction; 
cf.~(\ref{IsogTorRed}). Therefore 
$E$ is uniformizable by a torus: 
\begin{equation} \label{eq:unif.E}
0\to \G\to \C_K^\times \to E(\C_K)\to 0,  
\end{equation}
where $\G$, as a subgroup of $\C_K^\times$, is $q_E^\Z$ for some $q_E\in \C_K^\times$ with $\ord_K(q_E)>0$. 
More precisely, since $E$ carries a canonical principal polarization, it
is uniformized by the torus $\Hom(\Gamma,\C_K^\times)$; fixing a generator
$\rho$ of $\Gamma$, we identify 
$\Hom(\Gamma,\C_K^\times)$ with $\C_K^\times$ via the isomorphism
$f\mapsto f(\rho)$.  
By (\ref{eqFunct}), the closed immersion $\pi^\vee: E\to J$ induces a
homomorphism $\pi^\vee:\G\to \La$ and a homomorphism
of tori $\C_K^\times\to\Hom(\Lambda,\C_K^\times)$ making following diagram
commute:
\begin{equation} \label{eq:immersion.unif}
\xymatrix @=.25in{ 
  0 \ar[r] & \Gamma \ar[d]_{\pi^\vee} \ar[r] & 
  {\C_K^\times} \ar[r] \ar[d] &
  {E(\C_K)} \ar[r] \ar[d]^{\pi^\vee} & 0 \\
  0 \ar[r] & \Lambda \ar[r] & 
  {\Hom(\Lambda,\C_K^\times)} \ar[r] & {J(\C_K)} \ar[r] & 0
}\end{equation}
It is easy to see that the vertical arrows in~\eqref{eq:immersion.unif}
are injective.
In general, $\pi^\vee(\G)$ need not be saturated in $\La$, i.e., the abelian group $\La/\pi^\vee(\G)$ 
might have non-trivial torsion. Let $\G'$ be the saturation of $\pi^\vee(\G)$ in $\La$. 
We can write 
$$
\pi^\vee(\rho)=c\cdot \la_E, 
$$
where $c$ is a uniquely determined positive integer, $\la_E$ is a
generator of $\G'$, and $\rho$ is our fixed generator of $\Gamma$.
\end{prg}

\begin{prg}
  Let $\pi: \Lambda\to\Gamma$ be the homomorphism of character groups
  associated to the middle vertical arrow of~\eqref{eq:immersion.unif}.
  The homomorphism $\pi^\vee: \Gamma\to\Lambda$ induces the homomorphism of
  tori 
  $\ev_\rho:
  \Hom(\Lambda,\C_K^\times)\to\Hom(\Gamma,\C_K^\times)=\C_K^\times$ 
  given by $\ev_\rho(f) = f(\pi^\vee(\rho))$.
  By the discussion in~(\ref{eqFunct}), the following diagram commutes:
  \begin{equation} \label{eq:parameterization.unif}
    \xymatrix @=.25in{ 
      0 \ar[r] & \Lambda \ar[r] \ar[d]_\pi & 
      {\Hom(\Lambda,\C_K^\times)} \ar[r] \ar[d]_{\ev_\rho} & 
      {J(\C_K)} \ar[r] \ar[d]^\pi & 0 \\
      0 \ar[r] & \Gamma \ar[r] & 
      {\C_K^\times} \ar[r] &
      {E(\C_K)} \ar[r] & 0 
    }\end{equation}
  It is easy to see that the vertical arrows
  in~\eqref{eq:parameterization.unif} are surjective.
  \end{prg}
  
  \begin{prg}\label{prg-c}
  Let $c^{-1}\Gamma = \{ x\in\C_K^\times~|~x^c\in\Gamma \}$.  Since
  $\Gamma = q_E^\Z$ we have $c^{-1}\Gamma = \mu_c\times w^\Z$, where
  $\mu_c\subset\C_K^\times$ is the group of $c$-th roots of unity and 
  $w$ is any $c$-th root of $q_E$.  In particular, 
  \begin{equation}\label{eqOrdw}
    \ord_K(q_E)=c\cdot \ord_K(w). 
  \end{equation}
  Define
  $\ev_E:\Hom(\Lambda,\C_K^\times)\to\C_K^\times$ by
  $\ev_E(f) = f(\lambda_E)$.  Then $\ev_E^c = \ev_\rho$, so we
  have a commutative diagram
  \begin{equation} \label{eq:ev.E}
    \xymatrix @=.25in{ 
      0 \ar[r] & \Lambda \ar[r] \ar[d] & 
      {\Hom(\Lambda,\C_K^\times)} \ar[r] \ar[d]_{\ev_E} & 
      {J(\C_K)} \ar[r] \ar[d]^\pi & 0 \\
      0 \ar[r] & {c^{-1}\Gamma} \ar[r] & 
      {\C_K^\times} \ar[r] &
      {E(\C_K)} \ar[r] & 0 
    }\end{equation}
  where the map $\C_K^\times\to E(\C_K)$ in~\eqref{eq:ev.E}
  is the $c$-th power of the one in~\eqref{eq:parameterization.unif}.
  We claim that the vertical arrows in~\eqref{eq:ev.E} are again surjective.
  Since $\ev_E$ is surjective, by the snake lemma it suffices to prove that
  $\ker(\ev_E)\to\ker(\pi)$ is surjective.  Let $x\in\ker(\pi)$.  Since
  $\ker(\pi)$ is an abelian subvariety of $J$, it is divisible; choose 
  $y\in\ker(\pi)$ such that $cy = x$.  Since $\ker(\ev_\rho)$ surjects
  onto $\ker(\pi)$ there exists $z\in\ker(\rho)$ such that
  $z\mapsto y$.  Then $z^c\mapsto x$ and $\ev_E(z^c) = \ev_\rho(z)=1$,
  which proves the claim. This implies 
  \begin{equation} \label{eq:cinvGamma}
  c^{-1}\G=\{[\la, \la_E]\ |\ \la\in \La\}\subset K^\times. 
  \end{equation}
  In particular, $c$ divides the order of the group of roots of unity in $K$. 
\end{prg}

\begin{prg}
The endomorphism 
$$
e_0=\pi^\vee\circ\pi: J\to J
$$ 
corresponds to an idempotent $e\in \End^0(J):=\End(J)\otimes_\Z \Q$. 
Up to isogeny, we can decompose 
$$
J\sim_K A_1\times A_2\times\cdots \times A_s,
$$
where $A_i$'s are $K$-simple abelian varieties. This decomposition produces idempotents 
$$e_1, \dots, e_s \in \End^0(J)$$ which are mutually orthogonal: $e_ie_j=0$ if $i\neq j$. 
The idempotent $e$ is one of those. 
The $\Q$-bilinear form $B(x,y)=\Tr(xy^\dag)$ on $\End^0(A)$ is symmetric and 
positive definite (here the Rosati involution is with respect to the canonical principal polarization $H$). 
This implies that the Rosati involution must fix each idempotent $e_i$.  
Therefore $e^\dag=e$, and also $e_0^\dag=e_0$.  This observation will
simplify some calculations and is useful in the following paragraph.

We denote by $n$ the denominator of $e$ in $\End(J)$, i.e., the least 
natural number such that $ne\in \End(J)$. 
Note that (\ref{eqFunct}) implies that $\End(J)$ 
is naturally a subring of $\End(\La)$ when we regard $\Lambda$ as the
lattice uniformizing $J$, and $\End(J)$ is a subring of $\End(\La)^\opp$
when we regard $\Lambda$ as the character group of the torus uniformizing
$J$.  By Proposition~\ref{prop3.3} and the above discussion, the image of
$e_0$ in $\End(\La)$ is the same under either identification.
We define the denominator $r$ of $e$ in $\End(\La)$ as the least
natural number such that  
$re\in \End(\La)$. Obviously, $r$ divides $n$. 
\end{prg}

\begin{lem}\label{lemZagier}
The morphism $\pi\circ\pi^\vee: E^\ast\to E_\ast$ is the multiplication-by-$n$ map on $E$. 
\end{lem}
\begin{proof}
See the proof of Theorem 3 in \cite{Zagier}. 
\end{proof}

\begin{prg} 
  Recall that the closed immersion $E\inject J$ gives rise to the
  inclusion $\pi^\vee:\Gamma\inject\Lambda$ sending $\rho\mapsto c\lambda_E$,
  and that the projection $\pi:J\to E$ induces a surjective homomorphism
  $\pi:\Lambda\surject\Gamma$.  The endomorphism
  $\pi\circ\pi^\vee: E^\ast\to E_\ast$ corresponds to the endomorphism
  $\pi\circ\pi^\vee: \Gamma\inject\Lambda\surject\Gamma$, so 
  by Lemma~\ref{lemZagier},
  $\pi(c\Lambda_E) = \pi\circ\pi^\vee(\rho) = n\rho$,  and therefore
  \begin{equation} \label{eq:lambda.E.nc.rho}
    \pi(\lambda_E) = \frac nc \rho. 
  \end{equation}
  The idempotent $e_0$ corresponds to the composition
  $\pi^\vee\circ\pi:\Lambda\surject\Gamma\inject\Lambda$.
  We have 
  $\pi^\vee\circ\pi(\lambda_E) = \pi^\vee(\frac nc\rho) = n\lambda_E$,
  so $e_0 = ne$ because $e(\lambda_E) = \lambda_E$.
  Since $\frac 1c \pi^\vee(\Gamma)\subset\Lambda$ but
  $\frac 1{c'} \pi^\vee(\Gamma)\not\subset\Lambda$ for
  $c' > c$, we have $\frac 1c e_0\in\End(\Lambda)$ but 
  $\frac 1{c'} e_0\notin\End(\Lambda)$ for $c' > c$.
  Thus $re = \frac 1c e_0 = \frac nc e$, i.e.\
  \begin{equation} \label{eq:r.n.c}
    c = \frac nr
  \end{equation}
\end{prg}

\begin{prg}
  The pairing $\angles{\cdot,\cdot}$ coincides with the ($H$-polarized
  version of) Grothendieck's monodromy pairing; see~\cite[(14.2.5)]{SGA7} and~\cite[Thm. 2.1]{Coleman}.  
  By~\cite[(11.5)]{SGA7} the cokernel of the map
  $\Lambda\to\Hom(\Lambda,\Z)$ induced by the monodromy pairing
  $\angles{\cdot,\cdot}$ is naturally isomorphic to the component group
  $\Phi_J$.  The analogous statement holds for $E$,
  and we have a commutative diagram
  \begin{equation} \label{eq:comp.gps.monodromy}
    \xymatrix @=.25in{ 
      0 \ar[r] & \Lambda \ar[r]^(.3){\angles{\cdot,\cdot}} \ar[d]_\pi & 
      {\Hom(\Lambda,\Z)} \ar[r] \ar[d]_{\ev_\rho} & 
      {\Phi_J} \ar[r] \ar[d]^{\pi_*} & 0 \\
      0 \ar[r] & \Gamma \ar[r] & 
      {\Z} \ar[r] &
      {\Phi_E} \ar[r] & 0 
    }\end{equation}
  where $\ev_\rho(f) = f(\pi^\vee(\rho))$ as
  in~(\ref{eq:parameterization.unif}).  Since $\pi^\vee(\rho)=c\lambda_E$
  and $\Z\lambda_E$ is a direct summand of $\Lambda$, the cokernel of 
  $\ev_\rho$ is isomorphic to $\Z/c\Z$.  As $\pi:\Lambda\to\Gamma$ is
  surjective, this implies that
  \begin{equation} \label{eq:coker.compgps}
    \coker(\pi_*:\Phi_J\to\Phi_E) \cong \Z/c\Z. 
  \end{equation}
  This is a generalization of Formula 1 in \cite{RibetLetter}.  The following 
  corollary is also observed in \cite[Thm. 2]{TakahashiJP} in the context of Jacobians of Shimura curves. 
\begin{cor}\label{cor:cdivK}
$\#\coker(\pi_\ast)$ divides the order of the group of roots of unity in $K^\times$. 
\end{cor}  
\begin{proof}
Follows from (\ref{prg-c}) and (\ref{eq:coker.compgps}). 
\end{proof}
\end{prg}
  
  \begin{prg}
  The map $\Gamma\to\Z$ is the composition of
  $\Gamma\to K^\times$ with $\ord_K:K^\times\to\Z$; hence $\rho$ maps to 
  $\ord_K(q_E)$.  (This recovers the well-known fact that
  $\#\Phi_E = \ord_K(q_E)$.)  We have
  $\rho = \frac cn\pi(\lambda_E)$ by~\eqref{eq:lambda.E.nc.rho}, so since
  the left square commutes, 
  \[ c\angles{\lambda_E,\lambda_E} = \angles{\lambda_E,\pi^\vee(\rho)}
  = \frac nc\ord_K(q_E), \]
  and therefore,
  \begin{equation} \label{eq:c2.lambdaE.lambdaE}
    c^2\,\angles{\lambda_E,\lambda_E} = n\,\ord_K(q_E).
  \end{equation}
  This is essentially Formula 3 in \cite{RibetLetter}.
\end{prg}

\begin{prg} Let
\begin{align*}
m: &=\min\{\langle \la, \la_E\rangle>0\ |\ \la\in \La\};\\
\la_E^\perp: &=\{\la\in \La\ |\ \langle \la, \la_E\rangle=0\}. 
\end{align*}
The image of $\ev_\rho\circ\angles{\cdot,\cdot}$
in~\eqref{eq:comp.gps.monodromy} is exactly 
$\min\{\angles{\lambda,c\lambda_E}>0~|~\lambda\in\Lambda\} = c\cdot m$; as
$\pi:\Lambda\to\Gamma$ is surjective and the image of $\Gamma$ in $\Z$ is
generated by $\ord_K(q_E)$, this implies
\begin{equation} \label{eq:whats.m}
  c\cdot m = \ord_K(q_E).
\end{equation}

\end{prg}

\begin{thm} \label{thm:when.surj}
  The following are equivalent:
  \begin{enumerate}
  \item The functorially induced map on component groups
    $\Phi_J\to\Phi_E$ is surjective.
  \item $e_0$ is primitive in $\End(\Lambda)$.
  \item $c=1$.
  \item $n=r$.
  \item $\angles{\lambda_E,\lambda_E} = n\,\ord_K(q_E)$.
  \item $m = \ord_K(q_E)$.
\item $n=[\La: \la_E^\perp\oplus \Z\la_E]$. 
  \end{enumerate}
\end{thm}

\begin{proof}
We have (1) $\iff$ (3) by~\eqref{eq:coker.compgps}, 
(3) $\iff$ (4) by~\eqref{eq:r.n.c}, and (4) $\iff$ (2) since
$e_0 = ne$.  Conditions~(5) and~(6) are equivalent to~(3)
by~\eqref{eq:c2.lambdaE.lambdaE} and~\eqref{eq:whats.m}, respectively.
It is easy to see that $r=[\La: \la_E^\perp\oplus \Z\la_E]$, hence (4) $\iff$ (7). 
\end{proof}

\begin{rem}\label{remModDeg} Assume $X$ has a $K$-rational point. Fix such a point $P_0$, and 
let $\theta: X\hookrightarrow J$ be the Abel-Jacobi map which sends $P_0$ 
to the origin of $J$. Since $\theta(X)$ generates $J$, the composition $\pi\circ \theta$ 
gives a non-constant morphism $w: X\to E $. It is easy to show that 
the degree $\deg(w)$ of $w$ is $n$. The index $[\La: \la_E^\perp\oplus \Z\la_E]$ is the ``congruence number'' of $\la_E$ 
with respect to the monodromy pairing, i.e., is the largest integer $R_E$ such that there is an element 
in $\la_E^\perp$ congruent to $\la_E$ modulo $R_E$. 
Hence Theorem \ref{thm:when.surj} 
implies that $R_E$ divides $\deg(w)$  and the ratio is $c$. As we will show in Section \ref{sec2}, 
$n/R_E=c$ can be strictly larger than $1$. 
It is interesting to compare this fact with 
the relation between the degree of modular parametrization  
of an elliptic curve over $\Q$ and the congruence number of the corresponding newform. 

Let $E$ be an elliptic curve over $\Q$. One may view $E$ as an abelian variety 
quotient over $\Q$ of the modular Jacobian $J_0(N)$, where $N$ is the conductor of $E$. 
Assume $E$ is an optimal quotient of $J_0(N)$. The modular degree $n_E$ 
is the degree of the composite map $X_0(N)\to J_0(N)\to E$, where the second map is an optimal quotient, 
and the first map is the Abel-Jacobi map $X_0(N)\to J_0(N)$ sending the cusp $[\infty]$ to the origin. 
Let $S_2(N,\Z)$ be the space of weigh-$2$ cusp forms on $\G_0(N)$ with integer Fourier 
coefficients. Let $f_E\in S_2(N,\Z)$ be the newform attached to $E$. Let 
$R_E':=[S_2(N,\Z):f_E^\perp\oplus \Z f_E]$, where $f_E^\perp$ is the orthogonal 
complement of $f_E$ in $S_2(N,\Z)$ with respect to the Petersson inner product. 
In \cite{ARS}, the authors show that $n_E$ divides $R_E'$, but the ratio $R_E'/n_E$ 
can be strictly larger than $1$. 
\end{rem}

We use Theorem \ref{thm:when.surj} to give two conditions under which
$\Phi_J\to\Phi_E$ is surjective. 

\begin{lem}\label{lemGek} 
Let $\T$ be a commutative subring of $\End(J)$ 
with the same identity element and such that $e\in \T\otimes\Q$. 
Suppose there is a bilinear $\T$-equivariant pairing 
$$
(\cdot, \cdot):\ \T\times \La\to \Z 
$$ 
which is perfect if we consider $\T$ and $\La$ as free $\Z$-modules. Then
the equivalent conditions of Theorem~\ref{thm:when.surj} are satisfied.
\end{lem}
\begin{proof} 
Let $s$ be the denominator of $e$ in $\T$, i.e., the smallest positive integer such that 
$se\in \T$. Note that $r$ divides $n$ and $n$ divides $s$, since $\T\subseteq \End(J)\subseteq \End(\La)$. 
Let $\la\in \La$ be arbitrary, and denote $\la'=(re)\la\in \La$. 
Because $se\in \T$ is primitive, we can take it as part of a $\Z$-basis of $\T$. Now 
$$
(se, \la)=(1, (se)\la)=(1, \frac{s}{r} \la')=\frac{s}{r}(1, \la')\in \frac{s}{r}\Z.
$$
Hence $s/r$ divides the determinant of $(\cdot, \cdot)$ with respect to some  
$\Z$-bases of $\T$ and $\La$. The perfectness of the pairing is equivalent to this determinant 
being $\pm 1$. Therefore, $s=r$, which implies $r=n$. 
\end{proof}

\begin{prg} \label{prg:setup.lemCII}
We keep the notation of Lemma~\ref{lemGek}. 
As is easy to check, the assumption $e\in \T\otimes \Q$ implies that $\G'$ is $\T$-invariant, that is, 
for any $T\in \T$ we have $T\la_E=a(T)\cdot \la_E$ for some $a(T)\in \Z$. It is clear that the map $T\mapsto a(T)$ 
gives a homomorphism $\T\to \Z$. Denote the kernel of this homomorphism by $I_E$. 
Define 
$$
I_E\La=\{T\la\ |\ T\in I_E, \la\in \La\}=\{T\la-a(T)\la\ |\ T\in \T, \la\in \La\}. 
$$ 
Assume $a(T^\dag)=a(T)$ for all $T\in \T$. Since 
$$
\langle T\la-a(T)\la, \la_E\rangle=\langle \la, T^\dag\la_E\rangle -a(T)\langle \la, \la_E\rangle=0,
$$
we have an inclusion $I_E\La \subseteq \la_E^\perp$. Note that the index 
$[\lambda_E^\perp : I_E\Lambda]$ is finite since $1-e$ is the projection 
onto $\la_E^\perp\otimes \Q$. 
\end{prg}

\begin{lem}\label{lemCII} 
  The index $[\lambda_E^\perp : I_E\Lambda]$ is divisible by $c$.  In particular,
  if $I_E\La=\la_E^\perp$ then the equivalent conditions of
  Theorem~\ref{thm:when.surj} are satisfied. 
\end{lem}
\begin{proof}
For $T\in\T$ and $\lambda\in\Lambda$ we have
\[\begin{split}
[T\lambda - a(T)\lambda,\,\lambda_E] 
&= [T\lambda,\,\lambda_E]\,[\lambda,\,\lambda_E]^{-a(T)} \\
&= [\lambda,\,T^\dag\lambda_E]\,[\lambda,\,\lambda_E]^{-a(T)}
= [\lambda,\,a(T^\dag)\lambda_E]\,[\lambda,\,\lambda_E]^{-a(T)} = 1. 
\end{split}\]
Hence by~\eqref{eq:cinvGamma} we have a surjection
$[\cdot,\lambda_E]:\Lambda/I_E\Lambda \surject c^{-1}\Gamma\cong \mu_c\times w^\Z$.
Consider the short exact sequence
\begin{equation} \label{eq:lambda.perp.seq}
  0 \To \lambda_E^\perp/I_E\Lambda \To \Lambda/I_E\Lambda \To
  \Lambda/\lambda_E^\perp \To 0.
\end{equation}
Since $\Lambda/\lambda_E^\perp\cong\Z$, this identifies 
$\lambda_E^\perp/I_E\Lambda$ with the torsion part of
$\Lambda/I_E\Lambda$.  Since $\Lambda/I_E\Lambda$ surjects onto
$\mu_c\times w^\Z$, no non-torsion element of $\Lambda/I_E\Lambda$ maps
into $\mu_c$, so we must have $\lambda_E^\perp/I_E\Lambda\surject\mu_c$. 
\end{proof}


\section{Modular Jacobians}\label{sMJ} In this section we discuss Question \ref{qR} in the context 
of Jacobians of certain modular curves. 

\begin{prg}\label{prgJ0p}
Consider the modular curve $X_0(p)$ defined over $\Q$ classifying elliptic curves with cyclic subgroups of order $p$, where 
$p$ is prime. Assume the genus of $X_0(p)$ is not $0$; in particular, $p$ is odd. 
By a well-known result of Deligne and Rapoport, the  Jacobian $J_0(p)$ of $X_0(p)$ has good reduction over 
$\Q_\ell$ for any prime $\ell\neq p$, and has toric reduction over $\Q_p$;  cf. \cite[p. 288]{NM}. 

\begin{thm}\label{thmJ0p} Let $\pi: J_0(p)\to E$ be an optimal quotient defined over $\Q_p$, where $E$ is an elliptic curve.  
The induced map on component groups $\pi_\ast: \Phi_{J_0(p)}\to \Phi_E$ of the N\'eron models over $\Z_p$ 
is surjective. 
\end{thm}
\begin{proof}
This was proven by Mestre and Oesterl\'e; see \cite[Cor. 3]{MO}. A more general result was 
proven by Emerton in \cite{Emerton}. Both proofs rely 
on Ribet's level-lowering theorem \cite{RibetLL}, and the deepest results in \cite{MazurEisen}. 
We give a different proof, which uses Lemma~\ref{lemCII}. 

Since $J_0(p)$ is defined over $\Q$ and has semi-stable reduction, all its endomorphisms 
are defined over $\Q$; see \cite[Thm. 1.1]{RibetEnd}. This implies that $\pi$ and $E$ can be defined over $\Q$. 
Let $\T$ be the subring of $\End(J_0(p))$ generated 
by the Hecke operators $T_n$, $n\geq 1$ (see \cite[$\S$3]{RibetLL} for the definition).
If $f_E(z)=\sum_{n\geq 1}a_n e^{2\pi i z n}$ is the newform attached to
$E$ then one checks that $T_n\lambda_E = a_n\lambda_E$, which implies
$e\in\T\otimes\Q$.  
By \cite[p. 444]{RibetLL}, $T^\dag=w_p T w_p$ for $T\in\T$, where $w_p$ 
is the Atkin-Lehner involution of $J$. Since $w_p\la_E=\pm \la_E$, the condition $a(T)=a(T^\dag)$ 
of~(\ref{prg:setup.lemCII}) is satisfied.  Let $I_E$ be the kernel of the
map $\T\to\Z, T_n\mapsto a_n$.

The Jacobian $J_0(p)$ acquires split toric reduction over the 
unramified quadratic extension of $\Q_p$. Since $p$ is odd, Corollary~\ref{cor:cdivK} and (\ref{UnRam}) 
imply that $p$ does not divide $c = \#\coker(\Phi_{J_0(p)}\to\Phi_E)$. 
By Lemma~\ref{lemCII}, it is enough to show that for all $\ell\neq p$ 
such that $\Phi_E[\ell]\neq 0$ we
have $(\lambda_E^\perp/I_E\lambda)\tensor\F_\ell = 0$.  From the
sequence~\eqref{eq:lambda.perp.seq} we see that 
$\Lambda/I_E\Lambda \cong \Z\times(\lambda_E^\perp/I_E\lambda)$
as abelian groups, so it is enough to prove that 
$(\Lambda/I_E\Lambda)\tensor\F_\ell\cong\F_\ell$.  If 
$\fm_\ell = (I_E,\ell)\lhd\T$ then 
$(\Lambda/I_E\Lambda)\tensor\F_\ell = \Lambda/\fm_\ell\Lambda$.  When 
$\ell\neq 2$ or $\fm_\ell$ is Eisenstein, it is a consequence
of~\cite[Theorem~2.3]{RibetTorsion} that $\Lambda/\fm_\ell\Lambda\cong\F_\ell$. 

We claim that $\fm_\ell$ is Eisenstein when $\Phi_E[\ell]\neq 0$.
Considering the $\ell$-torsion subgroup $E[\ell]$ of $E$ as a $\Gal(\overline{\Q}/\Q)$-module, 
we obtain a representation $\rho:\Gal(\overline{\Q}/\Q)\to \GL_2(\F_\ell)$. 
This representation is isomorphic to the residual representation
$\rho_{\fm_\ell}$ attached to $\fm_\ell$; see \cite[$\S$5]{RibetLL}  
for the construction and properties of $\rho_{\fm_\ell}$. 
If $\sE$ is the N\'eron model of $E$ then since $\Phi_E[\ell]\neq 0$,
we have that $\sE[\ell]$ is a finite \'etale
group-scheme over $\Z_p$ which extends $E[\ell]$.
Therefore the Galois representation $\rho\cong\rho_{\fm_\ell}$ is finite,
so $\fm_\ell$ is Eisenstein by Proposition~2.2 in \cite{RibetTorsion}. 
\end{proof}

\end{prg}

\begin{prg}\label{prgBD} Let $D>1$ be a square-free integer divisible by an even number of primes, 
and $M\geq 1$ be a square-free integer coprime to $D$. Let $\G_0^D(M)$ be the group 
of norm-$1$ units in an Eichler order of level $M$ in the indefinite quaternion 
algebra $B$ over $\Q$ of discriminant $D$. Since $B$ is indefinite, by fixing an isomorphism 
$B\otimes\R\cong \M_2(\R)$, we can regard $\G_0^D(M)$ as a discrete subgroup of $\SL_2(\R)$. 
Let $X^D_0(M)=\G^D_0(M)\bs \cH$ be the 
associated Shimura curve, where $\cH=\{z\in \C\ |\ \im(z)>0\}$. This is a 
smooth projective curve, which has a canonical model over $\Q$. It is a moduli space of abelian surfaces 
equipped with an action of $B$ and $\G_0(M)$-level structure. 

The Jacobian $J^D_0(M)$ of $X^D_0(M)$ has toric reduction over $\Q_p$ 
if $p$ divides $D$; this follows from the work of Cherednik and 
Drinfeld (cf. \cite{BC}). Assume $\pi:J^D_0(M)\to E$ is an optimal quotient defined over $\Q$, where $E$ 
is an elliptic curve. 
Fix a prime $p$ dividing $D$, and let $\pi_\ast$ be the induced map on component groups of the corresponding N\'eron models over $\Z_p$. 
In the proof of Proposition 4.4 and Corollary 4.5 in \cite{BD}, Bertolini and Darmon 
implicitly assume that $c$ in the diagram (\ref{eq:ev.E}) with $J=J^D_0(M)$ is $1$. 
By Theorem \ref{thm:when.surj} this assumption is equivalent to $\pi_\ast$ being surjective. 
On the other hand, Question \ref{qR} in general has negative answer, so it is not clear 
whether the answer is always positive for the Jacobians of Shimura curves.  
In the positive direction, 
Takahashi proved that if the $\Gal(\overline{\Q}/\Q)$-module $E[\ell]$ is irreducible, then 
$\ell$ does not divide the order of the cokernel of $\pi_\ast$; see \cite[Thm. 1]{TakahashiJP}. 
The proof relies on the comparison of the degrees of different 
modular parametrizations of $E$ by both modular and Shimura curves. 
\end{prg}

\begin{rem} \label{rem:computations}
  Theorem~\ref{thm:when.surj} suggests a computational approach to finding
  an example of an optimal quotient $E$ of $J^D_0(M)$ such that
  the homomorphism $\pi_\ast$ of component groups is not surjective.  The
  computer algebra package Magma has an implementation of Brandt modules,
  which allows one to do calculations with the lattice $\Lambda$
  uniformizing the analytification of $J^D_0(M)$.  In particular, one can
  efficiently calculate the idempotent $e$.  The surjectivity question
  then reduces to whether or not $re$, as an endomorphism of
  $\Hom(\Lambda,K^\times)$, takes $\Lambda$ to itself.  This 
  calculation can in theory be carried out using
  $p$-adic $\Theta$-functions.
\end{rem}

\begin{prg}\label{prgDMC} 
Let $A=\F[T]$ be the ring of polynomials with coefficients in a finite field $\F$, 
and $F=\F(T)$ be the field of fractions of $A$. 
Let $K=\F(\!(1/T)\!)$ be the completion of $F$ at the place $1/T$, and $R$ the ring of integers of $K$. 
Let $\fn\lhd A$ be an ideal and 
$$
\G_0(\fn)=\left\{\begin{pmatrix} a & b \\ c & d\end{pmatrix}\in \GL_2(A)\ \big|\ c\in \fn\right\}. 
$$
The group $\G_0(\fn)$ acts discontinuously on the Drinfeld half plane $\Omega:=\C_K-K$, and 
the quotient $\G_0(\fn)\bs\Omega$ is the analytification of 
the Drinfeld modular curve $Y_0(\fn)$, which is a smooth affine algebraic curve defined over $K$. 
The $\C_K$-valued points of $Y_0(\fn)$ are in bijection with 
rank-$2$ Drinfeld $A$-modules over $\C_K$ with certain level structures. 
Let $J_0(\fn)$ be the Jacobian of the smooth projective 
curve containing $Y_0(\fn)$ as a Zariski dense subset. The Jacobian $J_0(\fn)$ 
has split toric reduction over $K$; cf. \cite[Thm. 2.10]{Analytic}. 

\begin{thm}\label{thmJDM} Assume $\pi: J_0(\fn)\to E$ is an optimal quotient defined over $K$, where $E$ is an elliptic curve. 
The induced map on component groups $\pi_\ast: \Phi_{J_0(\fn)}\to \Phi_E$ of the N\'eron models over $R$ 
is surjective. 
\end{thm}
\begin{proof}
The proof essentially consists of showing that the condition in Lemma \ref{lemGek} is satisfied. 
This heavily relies on the arithmetic theory of Drinfeld modular curves. 

There are Hecke operators defined in terms of correspondences on $Y_0(\fn)$ which 
generate a commutative $\Z$-subalgebra $\T$ of $\End(J_0(\fn))$; see \cite[$\S$1]{Analytic} for the definitions 
and basic properties. The Hecke algebra $\T$ also naturally act on the space of 
$\Z$-valued $\G_0(\fn)$-invariant harmonic cochains $H_!(\cT, \Z)^{\G_0(\fn)}$ 
on the Bruhat-Tits tree $\cT$ of $\PGL_2(K)$; again we refer to \cite[$\S$1]{Analytic} for the definitions. 
(The $\Z$-module $H_!(\cT, \Z)^{\G_0(\fn)}$ is the analogue in this context 
of $S_2(N, \Z)$ in Remark \ref{remModDeg}.) Let $\La$ be the uniformizing lattice of $J_0(\fn)$. 
The algebra $\T$ naturally acts on $\La$; cf. (\ref{eqFunct}). A crucial fact that we need 
is that there is a canonical $\T$-equivariant isomorphism between $\La$ and $H_!(\cT, \Z)^{\G_0(\fn)}$; 
see \cite[Thm. 1.9]{Analytic} and \cite[Lem. 9.3.2]{GR}. In \cite[Thm. 3.17]{Analytic}, Gekeler defines a 
bilinear $\T$-equivariant pairing 
\begin{equation*}
\T\times H_!(\cT, \Z)^{\G_0(\fn)}\to \Z
\end{equation*}
and proves that it is perfect after tensoring with $\Z[p^{-1}]$, where $p$ is the characteristic of $\F$.  
(This pairing is the function field analogue of the well-known perfect $\Z$-valued pairing between 
the Hecke algebra and $S_2(N,\Z)$; see \cite[Thm. 2.2]{RibetModp}.) Using the facts listed above, the argument in the proof of Lemma 
\ref{lemGek} shows that $c$ is a $p$-power. On the other hand, according to Corollary \ref{cor:cdivK}, 
$c$ divides $\#\F-1$, so $c$ is coprime to $p$. This implies that $c=1$. 
\end{proof}
\end{prg}

\section{Jacobians isogenous to a product of two elliptic curves}\label{sec2}

We start by giving a very explicit, equation-based, example. We will explain later in this section how this example 
can be obtained as a special case of a general construction. 

\begin{example}\label{eg2.1}
Let $K=\Q_p$, where $p$ is odd. Let $X$ be the hyperelliptic curve of genus $2$ with two 
affine charts $y^2=f(x)$ and $Y^2=g(t)$ glued in the obvious way, where 
$$
f(x)=\left(px^2+(p-1)\right)\left((p+1)x^2+p\right)\left(x^2+1\right),
$$
$Y=y/x^3$, $t=1/x$, and $g(t)=f(x)/x^6$.  These equations 
define the minimal regular model of $X$ over $K$. Indeed, modulo $p$, 
the equation $y^2=f(x)$ becomes $y^2=-x^2(x^2+1)$, which is a curve 
with singular point $(0,0)$. It is clear from the equation $y^2=f(x)$ that 
the maximal ideal $(x,y,p)$ is a regular point on this model. Similarly, on the other 
chart, we have in reduction $Y^2=-t^2(1+t^2)$, and the maximal ideal $(t, Y, p)$ 
is again regular. Hence, the model is regular, and has a special fibre consisting 
of an irreducible rational curve with two nodes. It follows from Example 9.2/8 in \cite{NM} 
that the Jacobian $J$ of $X$ has toric reduction over $K$, and Remark 9.6/12 in \textit{loc.\ cit.} implies
that $\Phi_J(\overline{\F}_p)=1$. 

Next, let $E$ be the elliptic curve given by the equation $y^2=x(x-1)(x+p)$. The $j$-invariant of $E$ 
has valuation $\ord_K(j)=-2$, so by the Tate algorithm $E$ has multiplicative reduction over $K$ and $\Phi_E(\overline{\F}_p)\cong \Z/2\Z$.  

There is a morphism $f:X\to E$ of degree $2$ given by 
$$
(x,y)\mapsto \left(p(p+1)x^2+p^2, p(p+1)y\right). 
$$
Let $\pi: J\to E$ be the homomorphism of the Jacobians induced by $f$ by the Albanese functoriality. 
Note that the induced map on component groups
$\pi_\ast:\Phi_J(\bar\F_p)\to\Phi_E(\bar\F_p)$ is not surjective.
We claim that $\pi$ is an optimal quotient. It is enough to prove this over $\bar{K}$. If 
$\pi$ is not optimal, then it factors as $J\to E'\xrightarrow{\varphi} E$, where $\varphi$ is an isogeny of degree $>1$ (defined over $\bar{K}$). 
But then $f$ factors through $X\to E'\xrightarrow{\varphi} E$. This is not possible since the degree of $f$ is $2$. 
\end{example}

\begin{prg}\label{prgExample} Let $c\geq 2$ be an integer dividing the order of the group of roots of unity in $K^\times$.  
Assume $c$ is coprime to the characteristic of the residue field $k$. 
Let $E_1$ and $E_2$ be two elliptic curves 
over $K$ with multiplicative reduction, which are not isogenous over the algebraic closure $\bar{K}$ of $K$. 
Assume $\Phi_{E_1}(\bar{k})\cong \Phi_{E_2}(\bar{k})\cong \Z/c\Z$; equivalently, the 
$j$-invariants of $E_1$ and $E_2$ have valuation $-c$. 
Assume $E_i[c](K)=E_i[c](\bar{K})\approx \Z/c\Z\times \Z/c\Z$ ($i=1,2$); this 
condition is automatic if $E_i$ has split multiplicative reduction and satisfies the previous assumption. 
\end{prg}

\begin{prg}
Let  
\begin{equation*}
e_c: E_i[c]\times E_i[c]\to \mu_c
\end{equation*}
be the Weil pairing. Recall that the Weil pairing is alternating, i.e., $e_c(P, P)=1$ for any $P\in E_i[c]$; cf. \cite[(2.8.7)]{KM}. 
There is a canonical subgroup of $E_i[c]$ corresponding to 
$(\sE_i^0)_k[c]\cong \Z/c\Z$. Fix a generator $g_i$ of this subgroup, and a generator $\zeta$ of $\mu_c$. 
Since $e_c$ is non-degenerate, we can find 
$h_i\in E_i[c]$ such that $E_i[c]\approx \langle g_i\rangle\times \langle h_i\rangle$, and 
$e_c(g_1, h_1)=e_c(g_2, h_2)=\zeta$. 
Let $\psi: E_1[c]\xrightarrow{\sim} E_2[c]$ be the unique isomorphism such that $\psi(g_1)=h_2$ and $\psi(h_1)=g_2$. 
This is an anti-isometry with respect to the $e_c$ pairings on $E_1[c]$ and $E_2[c]$ because 
\begin{equation*}
e_c(\psi(g_1), \psi(h_1))=e_c(h_2, g_2)=e_c(g_2, h_2)^{-1}=e_c(g_1, h_1)^{-1}. 
\end{equation*}
Let $A=E_1\times E_2$ and let $G\subset A[c]$ be the graph of $\psi$:
\begin{equation*}
G=\{(P, \psi(P))\ |\ P\in E_1[c]\}. 
\end{equation*}
The product of the canonical principal polarizations on $E_1$ and $E_2$ is a principal polarization $\theta$ on the 
product variety $A=E_1\times E_2$. 
\end{prg}
\begin{prop} There is a principal polarization on the quotient abelian variety $J:=A/G$
defined by $G$ and $\theta$. With this principal polarization, $J$ is isomorphic to the  
canonically principally polarized Jacobian variety 
of a smooth projective curve $X$ defined over $K$. The Jacobian $J$ has toric reduction. 
\end{prop}
\begin{proof}
The existence of $X$ follows from Theorem 3 in \cite{Kani}. It is important here that $\psi$ is an 
anti-isometry, and $E_1$ and $E_2$ are not isogenous. The curve $X$ can be defined over $K$ 
because $\psi$, by construction, is an isomorphism of Galois modules; cf. \cite[Prop. 3]{HLP}. The 
claim that $J$ has toric reduction follows from (\ref{IsogTorRed}). 
\end{proof}

\begin{lem}
$\Phi_J=1$. 
\end{lem}
\begin{proof} Clearly $G\subset A(K)$ is a subgroup isomorphic to $\Z/c\Z\times \Z/c\Z$. 
By (\ref{NerMod}), $G$ extends to a 
constant \'etale subgroup-scheme of $\sA$. The restriction to the closed fibre 
gives an injection $G\hookrightarrow \sA_k(k)$, which composed 
with $\sA_k\to \Phi_A$ gives a canonical homomorphism $\phi:G\to \Phi_A$. 
It is clear that $\Phi_A\cong \Phi_{E_1}\times \Phi_{E_2}$. 
Since 
$$
\sA_k^0[c]\cong \{(P_1, P_2)\ |\ P_i\in \langle g_i\rangle\}, 
$$ 
$G\cap \sA_k^0=0$. Therefore, $\phi$ is an isomorphism. Now Theorem 4.3 in \cite{PapikianJL} 
implies that $$\Phi_J\cong \Phi_A/G=1.$$ 
\end{proof}

\begin{lem}
$E_1$ and $E_2$ are optimal quotients of $J$. 
\end{lem}
\begin{proof}
Note that $E_i$ embeds into $J$ as a closed subvariety since $E_i\cap G =0$. The claim then 
follows from (\ref{IsogTorRed}). 
Alternatively, note that the quotient $J/E_1$ 
is isomorphic to $E_2/E_2[c]\cong E_2$, so, by definition, $E_2$ is an optimal quotient of $J$. 
\end{proof}

\begin{prg} In the special case when $c=2$, Proposition 4 in \cite{HLP} allows to compute an explicit equation for $X$ 
starting with equations for $E_1$ and $E_2$. Moreover, in this case the assumption that $E_1$ and $E_2$ 
are not isogenous can be relaxed to the assumption that $E_1$ and $E_2$ are not isomorphic over $\bar{K}$, i.e., 
have distinct $j$-invariants; see \cite[Thm. 3]{Kani}. With this in mind, consider the Legendre curves 
$$
E_1: y^2=x(x-1)(x-p) \quad \text{and}\quad E_2: y^2=x(x-1)(x+p) 
$$
over $\Q_p$, where $p$ is odd. These curves have distinct $j$-invariants, multiplicative reduction, 
$\Phi_{E_i}\cong \Z/2\Z$, and $E_i[2]$ is $\Q_p$-rational. (Note that $E_i$ 
has split multiplicative reduction if and only if $-1$ is a square modulo $p$.)

Let $P_1=(1,0)$, $P_2=(0,0)$ and $P_3=(p,0)$ be the non-trivial elements of $E_1[2]$. 
Similarly, let $Q_1=(1,0)$, $Q_2=(0,0)$ and $Q_3=(-p,0)$ be the non-trivial elements of $E_2[2]$. 
Modulo $p$ the point $P_1$ lies in the smooth locus of the reduction of $E_1$, 
hence its specialization lies in the connected component $\sE_1^0$ of the identity. Define $\psi$ by 
$$
\psi(O)=O, \quad \psi(P_1)=Q_2, \quad \psi(P_2)=Q_1, \quad \psi(P_3)=Q_3. 
$$
Using the formulas in \cite[Prop. 4]{HLP}, one obtains the equation in Example \ref{eg2.1}. 
\end{prg}

\begin{rem} When $c\geq 3$, it seems rather difficult to write down an explicit equation for $X$. 
Below we will compute the 
$p$-adic periods of $J$ from the Tate periods of $E_1$ and $E_2$. In \cite{Teitelbaum}, Teitelbaum  
developed a method for computing an equation for a genus $2$ curve $X$ 
with split degenerate reduction from the periods of its Jacobian. Teitelbaum's formulae are $p$-adic, 
i.e. the coefficients of the equation of $X$ are given by infinite series.  
\end{rem}

In order to illustrate the machinery of Section~\ref{sec3}, 
we give an analytic interpretation of our previous algebraic construction, with some 
generalizations.  

\begin{prg} \label{analyticJ}
  Let $\fT = (\gm_{m,K}^2)^\an$ be a two-dimensional split analytic torus
  over $K$.  Fix $q_1,q_2\in K^\times$ such that
  $\ord_K(q_1),\ord_K(q_2)>0$ and $q_1^u\neq q_2^w$ for any
  non-zero $u,w\in\Z$.  Let $c > 1$ be an integer and let 
  $\zeta\in K^\times$ be a $c$-th root of unity.
  Let $\Lambda\subset\fT(K) = (K^\times)^2$ be the free abelian group
  generated by $(q_1,\zeta)$ and $(\zeta,q_2)$.  We have
  $\trop(q_1,\zeta) = (-\log|q_1|,0)$ and
  $\trop(\zeta,q_2) = (0,-\log|q_2|)$ which are linearly independent in 
  $\R^2$, so $\Lambda$ is a lattice in $\fT$.  Let
  $J^\an$ be the analytic quotient $\fT/\Lambda$.
\end{prg}

\begin{prg}\label{prg5.2}
  We identify $(n_1,n_2)\in \Z^2$ with the character of $\fT$ defined by
  $(Z_1,Z_2)\mapsto Z_1^{n_1}Z_2^{n_2}$.  Define 
  $H: \Lambda\isom\Z^2$ by 
  \[ H(q_1,\zeta) = (1,0) \quad\text{ and }\quad
  H(\zeta,q_1) = (0,1). \]
  We have
  \[ H(q_1,\zeta)(q_1,\zeta) = q_1 \quad
  H(q_1,\zeta)(\zeta,q_2) = \zeta =
  H(\zeta,q_2)(q_1,\zeta) \quad
  H(\zeta,q_2)(\zeta,q_2) = q_2, \]
  so $H(\lambda)(\mu) = H(\mu)(\lambda)$ for all $\lambda,\mu\in\Lambda$.
  Moreover, the symmetric bilinear form 
  $\angles{\cdot,\cdot}_H$ has the matrix form 
  $\smallmat{\ord_K(q_1)}00{\ord_k(q_2)}$ with respect to the above choice
  of basis, so $\angles{\cdot,\cdot}_H$ is positive definite.  Therefore
  by Theorem~\ref{thmRF}, $J^\an$ is the analytification of an abelian
  variety $J$, and the Riemann form $H$ gives rise to a principal
  polarization of $J$ by (\ref{polarization}).
\end{prg}

\begin{prg}
  By an \emph{elliptic subvariety} of $J$ we will mean an abelian
  subvariety $E$ of $J$ of dimension one.  By (\ref{IsogTorRed}), any
  elliptic subvariety of $J$ has split multiplicative reduction; moreover,
  if $0\to\Gamma\to\C_K^\times\to E(\C_K)\to 0$ is the Tate uniformization
  of $E$ then we have a homomorphism of short exact sequences
  \begin{equation} \label{eq:elliptic.subvar} \xymatrix @=.25in{
    0 \ar[r] & \Gamma \ar[r] \ar[d] & {\C_K^\times} \ar[r] \ar[d]^\varphi & 
    {E(\C_K)} \ar[r] \ar[d] & 0 \\
    0 \ar[r] & \Lambda \ar[r] & {(\C_K^\times)^2} \ar[r] & 
    {J(\C_K)} \ar[r] & 0
  }\end{equation}
  with injective vertical arrows.  In particular, 
  $\varphi(\C_K^\times)\cap\Lambda = \varphi(\Gamma)$.  Conversely, let
  $\gm_{m,K}^\an\cong\fT'\subset\fT$ be a subtorus of dimension one such
  that $\Gamma = \fT'(K)\cap\Lambda\cong\Z$ (equivalently, such that
  $\fT'(\C_K)\cap\Lambda\neq\{1\}$), and let $E^\an = \fT'/\Gamma$.  Then
  $E^\an$ is the analytification of an elliptic curve $E$ over $K$ and the
  induced map $E\to J$ is a closed immersion, so $E$ is an elliptic subvariety of
  $J$ and the diagram (\ref{eq:elliptic.subvar}) commutes.
\end{prg}

\begin{prop} \label{prop:elliptic.subvars}
  Let $J$ be as in (\ref{analyticJ}).  There are exactly two elliptic
  subvarieties of $J$, given by
  \[ E_1(\C_K) = \C_K^\times\times\{1\}/(q_1^c,1)^\Z \quad\text{ and }\quad
  E_2(\C_K) = \{1\}\times\C_K^\times/(1,q_2^c)^\Z. \]
\end{prop}
\begin{proof}
  It is clear that $E_1$ and $E_2$ are elliptic subvarieties of $J$.  Any
  dimension-one subtorus $\fT'$ of $\fT$ is of the form
  \[ \fT'(\C_K) = \{ (z,w) ~|~ z^\alpha w^\beta = 1 \} \]
  for some coprime integers $\alpha,\beta\in\Z$.  Let $\fT'$ be such a
  subtorus, and suppose that $\fT'(K)\cap\Lambda\neq\{1\}$.  Let
  $\lambda\in\Lambda\setminus\{1\}$ be an element of
  $\fT'(K)\cap\Lambda$.  Then 
  \[ \lambda = (q_1,\zeta)^\gamma(\zeta,q_2)^\delta
  = (q_1^\gamma\zeta^\delta,\, q_2^\delta\zeta^\gamma) \]
  for some integers $\gamma,\delta$, not both equal to zero, and we have
  \[ q_1^{\alpha\gamma}q_2^{\beta\delta}\zeta^{\alpha\delta+\beta\gamma}
  = 1. \]
  Raising both sides to the $c$-th power gives
  $q_1^{\alpha\gamma c}q_2^{\beta\delta c} = 1$, so we must
  have $\alpha\gamma = \beta\delta = 0$ by the way we chose $q_1,q_2$.
  If $\alpha\neq 0$ and $\beta\neq 0$ then $\gamma=\delta=0$, which
  contradicts our choice of $\lambda$.  Hence either $\alpha = 0$ and
  $\beta=\pm 1$, in which case $\fT'(\C_K) = \C_K^\times\times\{1\}$, or
  $\beta = 0$ and $\alpha = \pm 1$, in which case
  $\fT'(\C_K) = \{1\}\times \C_K^\times$.
\end{proof}

\begin{prg}
  Let $\Lambda'$ be the sublattice of $\Lambda$ generated by $(q_1^c,1)$ and
  $(1,q_2^c)$.  Identify $E_1$ (resp.\ $E_2$) with $\C_K^\times/q_1^{c\Z}$
  (resp.\ $\C_K^\times/q_2^{c\Z}$) in the obvious way.
  Let $A = E_1\times E_2$, so 
  $A(\C_K) = (\C_K^\times)^2/\Lambda'$, and the kernel of the
  multiplication map $A\to J$ is $\Lambda/\Lambda'\cong(\Z/c\Z)^2$.
  Since $E_1$ and $E_2$ are the only elliptic subvarieties of $J$, it
  follows that $J$ is not isomorphic to a product of elliptic curves.
  Therefore the theta divisor of $J$ is a smooth curve $X$ of genus $2$,
  and $J$ is isomorphic to the Jacobian of $X$ as principally polarized
  abelian varieties.
\end{prg}

\begin{prg}
  Since $E_1$ and $E_2$ are subvarieties of $J$, for $i=1,2$ the dual
  homomorphism $J\to E_i$ is an optimal quotient by (\ref{IsogTorRed}).  
  Let $\Gamma_1 = (q_1^c,1)^\Z$ and $\Gamma_2 = (1,q_2^c)^\Z$, and for
  $i=1,2$ let $\Gamma_i'$ be the saturation of $\Gamma_i$ in $\Lambda$.
  Then $\Gamma_1' = (q_1,\zeta)^\Z$ and $\Gamma_2' = (\zeta,q_2)^\Z$.  
  It follows from~\eqref{eq:coker.compgps} that the cokernel of the map on
  component groups $\Phi_J\to\Phi_{E_i}$ is isomorphic to $\Z/c\Z$.  In
  particular, $\Phi_J\to\Phi_{E_i}$ is not surjective. 
  Note that the image of $\La$ in $\C_K^\times$ under the evaluation map $\ev_{E_i}$ 
  is generated by $\zeta$ and $q_i$ -- this is immediate from the definition of $H$ in (\ref{prg5.2}). 
  This illustrates the surjectivity of 
  the map $\Lambda\to c^{-1}\Gamma_i$ of~\eqref{eq:ev.E}.

\end{prg}

\begin{prg}
  A calculation involving $p$-adic $\Theta$-functions shows that the
  Weil pairing on the $c$-torsion of the Tate curve $E_i$ is given by the rule
  $e_c(\zeta,q_i) = \zeta$.  Note that $\zeta\in E_i$
  generates the subgroup of $E_i[c]$ which reduces to the identity
  component of the N\'eron model of $E_i$.  Let $\psi:E_1[c]\to E_2[c]$ be
  the unique isomorphism such that 
  $\psi(\zeta) = q_2$ and $\psi(q_1) = \zeta$.  Then the graph
  \[ G=\{(P, \psi(P))\ |\ P\in E_1[n]\} \]
  is exactly the kernel of the map $A = E_1\times E_2\to J$, so this
  analytic construction coincides with our algebraic construction, 
  at least when $c = \ord_K(q_1) = \ord_K(q_2)$.
\end{prg}

\begin{prg}
  Let $E = E_1$.  In the notation of Section~\ref{sec3}, we have
  $q_E = q_1^c$, so $\ord_K(q_E) = c\cdot\ord_K(q_1)$.  We can take
  $\lambda_E = (q_1,\zeta)\in\Lambda$, so 
  \[ \angles{\lambda_E,\lambda_E} 
  = \ord_K H(q_1,\zeta)(q_1,\zeta) = \ord_K(q_1). \]
  It is clear that 
  $m = \min\{\angles{\lambda,\lambda_E}>0~|~\lambda\in\Lambda\}$
  is equal to 
  $\angles{\lambda_E,\lambda_E} = \ord_K(q_1) = \ord_K(q_E)/c$.  
  Hence $c\angles{\lambda_E,\lambda_E} = \ord_K(q_E)$, so
  $c = n$ by~\eqref{eq:c2.lambdaE.lambdaE}, and hence $r=1$
  by~\eqref{eq:r.n.c}.  
  The fact that $r=1$ is easy to see directly, as the
  idempotent $e$ corresponds to the endomorphism $(a,b)\mapsto(a,0)$ of
  the character group $\Z^2\cong\Lambda$ of $\fT$, so $e\in\End(\Lambda)$.
  The equality $n = c$ is then clear as well since the smallest power of
  the endomorphism $(x,y)\mapsto(x,1)$ of $\gm_{m,K}^2$ sending $\Lambda$
  to itself is $c$.  

\end{prg}


\bibliographystyle{amsplain}
\bibliography{Example-13-08-31}

\end{document}